\documentclass{amsart}
\usepackage{amsmath,amssymb,amscd,amsfonts}
\usepackage{sty-lcc,sty-ch}

\usepackage{todonotes}

\renewcommand\eqref[1]{(\ref{#1})}

\title{Non-rational centers of log canonical singularities}
\author{Valery Alexeev}
\address{Department of Mathematics, University of Georgia, Athens, GA 30605, USA}
\email{valery@math.uga.edu}
\author{Christopher D. Hacon}
\address{Department of Mathematics, University of Utah, 155 South 1400 East,
Salt Lake City, UT 48112-0090, USA}
\email{hacon@math.utah.edu}

\begin{document}
\begin{abstract}
  We show that if $(X,B)$ is a log canonical pair with $\dim X\geq
  d+2$, whose non-klt centers have dimension $\geq d$, 
  then $X$ is has depth $\ge d+2$ at every closed point.
\end{abstract}

\maketitle
\tableofcontents

\section{Introduction}

The main classes of singularities in the Minimal Model Program are:
terminal, canonical, log terminal, and log canonical. It is well known
that the singularities in the first three classes are Cohen-Macaulay
and rational, and in the last class they are neither, in general. The
main aim of this paper is to establish several general tools for measuring how
far a particular log canonical singularity is from being CM
(Cohen-Macaulay) or rational. (We note that \cite{KK10} takes a
different approach, and proves that all log canonical
singularities are Du Bois.)  

Let $X$ be an algebraic variety over an algebraically closed field of
characteristic zero and let $f:Y\to X$ be a resolution of
singularities. It is well known that the sheaves $R^if_*\cO_Y$ 
are coherent on $X$ and do not depend on the choice of the resolution. We make
the following definition.

\begin{definition}
  Let $X$ be a normal algebraic variety.
  The \defn{centers of non-rational singularities} of $X$
  (or simply \defn{non-rational centers}) are 
the subvarieties $Z_i$ defined by the
  associated primes of the sheaves $R^if_*\cO_Y$, $i>0$. 
\end{definition}

Thus, each non-rational center $Z_i$ is an irreducible subvariety of
$X$, by definition.

Let $(X,B)$ be a klt pair (see Section~\ref{sec:preliminaries} below
for all standard definitions). Then $R^if_*\cO_Y=0$ for $i>0$ and $X$
has rational singularities, hence $X$ is CM.  As mentioned above, log
canonical does not imply rational or CM: the simplest example is
provided by the cone over an abelian surface in which case
$R^1f_*\cO_Y$ is supported at the vertex and $X$ is also not CM.  Our
first main theorem is the following:

\begin{theorem}\label{thm:centers}
  Let $(X,B)$ be a log canonical pair. Then every non-rational center
  of $X$ is a non-klt center of $(X,B)$.
\end{theorem}

\begin{remark} A similar result was independently proven in
  \cite{Kov11}. Note that in \cite{Kov11} the terminology "irrational
  centers" is used instead of non-rational centers.
\end{remark}
Note that the the closed set of non-rational
singularities is a subset of the closed set of non-klt singularities,
but \eqref{thm:centers} is far from being obvious.

It is natural to assume that the failure of log canonical pairs to be
CM can be described in terms of the non-rational centers. There are
several ways to measure this failure.  A variety is CM if and only if it
satisfies Serre's condition $S_{\dim X}$ or, equivalently, if and only if it
is $S_{\dim X}$ \emph{at every closed point.} Thus, the conditions
$S_n$ 
generalize the CM property. But there is another
logical generalization:

\begin{definition}\label{def:condition-C}
  We say that a coherent sheaf $F$ on a scheme $X$ satisfies condition
  $C_n$ (or simply is $C_n$) if for every \emph{closed point}
  $x\in\Supp F$ one has $\depth F_x\ge n$. (Here, ``C'' stands for a
  ``closed point''.) If $F=0$ then we say that $F$ is $C_n$ for all~$n$.
  We say that $X$ is $C_n$ if so is $\cO_X$.
\end{definition}

It turns out that for projective varieties the $C_n$ condition is
frequently easier to work with than the $S_n$ condition because it
admits a simple cohomological criterion, see
Lemma~\ref{lem:Hartshorne}.  Our second main result is the following
theorem, generalizing the $C_3$ case contained in \cite[3.1 and
3.2]{AlexeevLimits}. (Recently, O. Fujino communicated to us another
proof of the $C_3$ case \cite[4.21 and 4.27]{Fujino09}).

\begin{theorem}\label{thm:main}
  Let $X$ be a normal variety of $\dim X\ge d+2$.  Assume
  that the pair $(X,B)$ is log canonical and that every non-klt 
  center of $(X,B)$ 
has dimension $\ge d$. Then
  \begin{enumerate}
  \item For each $i>0$, the sheaf $R^i f_*\cO_Y$ is $C_{d+1-i}$.
  \item $X$ is $C_{d+2}$.
  \end{enumerate}
\end{theorem}

\begin{remark}It would be interesting to know if a similar result is true when
``non-klt centers'' are replaced by ``non-rational centers''.\end{remark}

\begin{acknowledgments}
  The first author was partially supported by NSF research grant DMS
  0901309, the second author was partially supported by NSF research
  grant DMS 0757897. Some of the work on this paper was done at MSRI
  and we would like to thank MSRI for the hospitality. We also thank
  Florin Ambro, Yujiro Kawamata and Karl Schwede for some useful
  discussions. We also thank the referee for helpful suggestions.
\end{acknowledgments}

\section{Preliminaries}
\label{sec:preliminaries}

We work over the field of complex numbers $\bC$.

\subsection{Singularities of the MMP and non-klt centers}
\label{sec:klt-lc}
 Let $X$ be a normal
variety. A {\bf boundary} is a $\bQ$-divisor $B=\sum b_iB_i$ on $X$
such that $0\le b_i\le 1$. 
If $K_X+B$ is $\bQ$-Cartier, then $(X,B)$ is a {\bf log pair}.  A {\bf
  log resolution} of a log pair $(X,B)$ is a projective birational
morphism $f:Y\to X$ such that $Y$ is smooth and ${\rm Exc}(f)\cup
f^{-1}(B)$ is a divisor with simple normal crossings.  We may then
uniquely write $$K_Y+f^{-1}_*B\equiv f^*(K_X+B)+\sum a_E(X,B)E$$ where
$E\subset Y$ are all the $f$-exceptional divisors.  The numbers
$a_E(X,B)$ are the {\bf discrepancies} of $(X,B)$ along $E$.  They do
not depend on the choice of a log resolution $f$.  A pair $(X,B)$ is
\begin{enumerate}
\item {\bf log canonical}, abbreviated {\bf lc}, if $b_i\leq 1$ and
  $a_E(X,B)\geq -1$, 
\item {\bf kawamata log terminal (klt)} if $b_i< 1$ and $a_E(X,B)>
  -1$ 
\end{enumerate}
for some (or equivalently for all) log resolution.

A {\bf non-klt place} is a component of $\rdown B .$ or a divisor of
discrepancy $a_E(X,B)\leq -1$. A {\bf non-klt center} is the image in $X$ of
a non-klt place (note that often in the literature non-klt places and
centers are called lc places and centers).  

A pair $(X,B)$ such that $b_i\leq 1$ is {\bf divisorially log
  terminal} (dlt for short) if there is a log resolution $f:Y\to X$
such that $a_E(X,B)> -1$ for any $f$-exceptional divisor $E\subset Y$.
Equivalently, by \cite{Szabo} $(X,B)$ is dlt if there is an open
subset $U\subset X$ such that $U$ is smooth, $B|_U$ has simple normal
crossings and $a_E(X,B)> -1$ for any divisor $E$ over $X$ with center
contained in $X-U$.

\subsection{CM, $S_n$, and $C_n$}
\label{sec:cm-sn-cn}

\begin{definition}
  A coherent sheaf $F$ on a Noetherian scheme $X$ is Cohen-Macaulay if
  for every scheme point $x\in \Supp F$ one has $\depth F_x = \dim_x
  \Supp F$. The sheaf $F$ satisfies Serre's condition $S_n$ if one has
  $\depth F_x\ge \min(n, \dim_x \Supp F)$.
\end{definition}

Note that some authors (e.g. \cite{KollarMoriBook}), define $S_n$ by
the condition $\depth F_x\ge \min(n, \dim_x X)$. We follow \cite[5.7.1
and 5.7.2]{EGA4}. This should not lead to confusion in the
settings of this paper.

Compare the condition $S_n$ with our condition $C_n$. One obvious
difference is that in Definition~\ref{def:condition-C} we did not ask
for $\min(n,\dim_x\Supp F)$. Hence, a CM variety satisfies $S_{\dim X+1}$
but not $C_{\dim X+1}$. A more subtle difference is provided by the
following example.

\begin{example}
  Let $Y$ be the cone over an abelian surface and let
  $X=Y\times\bP^1$. Then $X$ is not $S_3$ at the generic point of
  $Z=\rm{(vertex)}\times\bP^1$.  However, $X$ is $S_3$ at every closed
  point $x\in X$: $\cO_{X,x}$ is $S_3$ if and only if the hyperplane section
  $\cO_{Y,x}$ is $S_2$, which is true since $Y$ is normal.  Thus $X$
  is $C_3$ but not $S_3$.
\end{example}

Thus, assuming $n\le\dim X$, the property $S_n$ is stronger than
$C_n$. On the other hand, knowing that a certain class of varieties
satisfies $C_n$ allows to conclude the property $S_n$ for this class
indirectly, as follows.  (See e.g.  \cite{AlexeevLimits} for an
application of this principle.)  If $X$ is not $S_n$ at the generic
point of a subvariety $Z$ then a general hyperplane section $H$ is not
$S_n$ at the generic points of $Z\cap H$. Cutting down this way, we
get to a variety which is not $C_n$, and this process preserves other
nice properties of $X$, such as being log canonical.

The reason why the condition $C_n$ is so convenient to work with for
projective varieties is the following simple cohomological criterion:

\begin{lemma}\label{lem:Hartshorne}
  Let $F$ be a coherent sheaf on a projective scheme $X$ with an ample
  invertible sheaf $L$, over an algebraically closed field.  Then $F$
  is $C_n$ if and only if $H^i(X, F(-sL))=0$ for any $i<n$ and $s\gg0$.
\end{lemma}
\begin{proof}
  This is basically proved in
  \cite[III.7.6]{HartshorneAlgebraicGeometry}, although it is not
  stated there in this way. We give the proof here for clarity.

  We embed $X$ into $P=\bP^N$ by some power of $L$. 
  Then the cohomology group $H^i(X,F(-q))=H^i(P,F(-q))$ is dual to
  the group $\Ext^{N-i}_P(F,\omega_P(q))$ which equals
  $\Gamma(P,\itExt_P^{N-i}(F,\omega_P(q)))$ if $q\gg0$ (where
  $\omega_P=\cO_P(-N-1)$ is the dualizing sheaf). Thus,
  one has $H^i(X,F(-sL))=0$ for $i<n$ and $s\gg0$ if and only if the sheaf
  $\itExt_P^{N-i}(F,\omega_P)$ is zero for $i<n$ or, equivalently, if
  $\itExt_P^{N-i}(F,\cO_P)=0$ for $i<n$. 

  This sheaf is zero if and only if its stalks are zero at every closed point
  $x\in \Supp F$. Denote $A=\cO_{P,x}$ for short, it is a 
  regular local ring.  The stalk at $x$ is $\Ext^{N-i}_{A}(F_x,A)$ and it is zero
  for $i<n$ if and only if the projective dimension $\pd_A F_x\le N-n$ (see
  \cite[Ex.6.6]{HartshorneAlgebraicGeometry}). The latter is
  equivalent to $\depth F_x\ge n$ since $\pd_A F_x + \depth F_x =\dim A=N$
  (cf. \cite[6.12A]{HartshorneAlgebraicGeometry}).
\end{proof}

\begin{example}
  Let $Y$ be a projective klt variety such that
  $NK_Y \sim 0$ for some $N\in \bN$. Let $L$ be an ample invertible
  sheaf on $Y$. Then the cone $X=\Spec \oplus_{k\ge0} H^0(Y,L^k)$ is
  lc, and its vertex $P$ is the unique non-klt center. Indeed, for the
  blowup $f:X'\to X$ at $P$ which inserts an exceptional divisor $E=Y$,
  one has $K_{X'}+E=f^*K_X$, $(X',E)$ is dlt, and so $E$ is the only
  divisor with discrepancy $a_E(X,B)=-1$.

  Assuming $n\le \dim X$, the cone $X$ satisfies condition $C_n$ if and only if
  $H^i(\cO_Y)=0$ for $0<i\le n-2$. This follows from the interpretation
  of the depth in terms of the local cohomology groups $H_P^i(\cO_X) =
  \oplus_{k\in \bZ}H^i(\cO_Y(L^k))$ and the fact that 
  $H^i(\cO_Y(L^k))=0$ for $k\ne0$ by the Kawamata-Viehweg vanishing
  and Serre duality.

  In particular, if $Y$ is an abelian surface then $X$ is not $C_3$ at
  the vertex $P$, but if $Y$ is a K3 or Enriques surface then $Y$
  \emph{is} $C_3$. Since for $i>0$ one has  $R^i f_*\mathcal O_{X'} =
  H^i(\cO_Y)$, we see that $P$ is a non-rational center if $Y$ is
  abelian or K3, and there are no non-rational centers if $Y$ is an
  Enriques surface.
\end{example}

By a theorem of Kempf, a variety $X$ has rational singularities if and only if
$X$ is CM and $f_*\omega_Y=\omega_X$ (see e.g.
\cite[5.12]{KollarMoriBook}). We conclude this section
with the following result which we will need below.

\begin{theorem}\label{t-klt} Let $(X,B)$ be a dlt pair, then $X$
has rational singularities.
\end{theorem}
\begin{proof} \cite[5.22]{KollarMoriBook}.
\end{proof}

\subsection{Ambro's and Fujino's results on quasi log varieties}
\label{sec:qlc-varieties}

In this section we recall some definitions and results concerning
quasi-log-canonical pairs (qlc pairs for short).
\begin{definition} Let $Y\subset M$ be a simple normal crossings
  (reduced) divisor $Y$ on a smooth variety and $D$ be a divisor on
  $M$ whose support contains no components of $Y$ and such that $D+Y$
  has simple normal crossings support.  The pair $(Y,B=D|_Y)$ is a
  {\bf global embedded simple normal crossings pair}.  Let $\nu :Y^\nu
  \to Y$ be the normalization and $K_{Y^\nu}+\Theta =\nu ^*(K_Y+D)$.
  A {\bf stratum} of $(Y,B)$ is a component of $Y$ or the image of a
  non-klt center of $(Y^\nu , \Theta )$.
Thus, the strata of $(Y,B)$ are irreducible by definition.
  
\end{definition}

We have the
following Torsion-Freeness and Vanishing Theorems of F. Ambro cf.
\cite[2.47]{Fujino09}:

\begin{theorem}\label{t-tor} Let $(Y,B)$ be a global embedded simple normal crossings pair. Assume that $B$ is a boundary $\Q$-divisor $L$ is a Cartier divisor and $f:Y\to X$ is a proper morphism. Then
\begin{enumerate}
\item If $L-(K_Y+B)$ is semiample over $X$, then every non-zero local section
of $R^q f_* \OO _Y(L)$ contains in its support the image of some
stratum of $(Y,B)$.

\item If $\pi:X\to Z$ is a projective morphism and there is a
  $\bQ$-Cartier divisor $H$ on $X$ such that $f^*H\sim _\Q L-(K_Y+B)$,
  and such that $H$ is big and nef on the image of every stratum of
  $(Y,B)$, then $R^p\pi _*(R^q f_* \OO _Y(L))=0$ for any $p>0$ and
  $q\geq 0$.
\end{enumerate}
\end{theorem}\begin{proof}
\cite[3.2]{AmbroQLvars}, 
\cite[2.47]{Fujino09}.
\end{proof}

\begin{corollary}\label{cor:assoc-primes}
  Let $\OO_Y(L)$ be as in part (1) of the above Theorem \ref{t-tor}.
  Then for any $q\ge0$, each associated prime of the sheaf $R^q f_* \OO
  _Y(L)$ is the generic point of the image of a stratum of $(Y,B)$.

\end{corollary}
\begin{proof}
  Since the question is local on the base, we can assume that $X=\Spec
  R$ is affine. Let $P\subset R$ be an associated prime of
  $R^qf_*\OO_Y(L)$. By definition, there exists a nonzero 
  section $s\in \Gamma(X,R^qf_*\OO_Y(L))$ whose support is $Z(P)$.  

  We claim that $Z(P)$ is the image of a stratum of
  $(Y,B)$. Suppose it is not. Let $Z(Q_i)$ be all the (finitely many)
  images of strata of $(Y,B)$ that are contained in
  $Z(P)$. We have $Z(Q_i) \neq Z(P)$. Over the open subset $U=X
  \setminus \cup Z(Q_i)$ the support of the section $s|_U$ is
  $Z(P)\cap U\ne\emptyset$. But by \ref{t-tor}(1) we must have
  $s|_U=0$. Contradiction.
\end{proof}

We will also need the following weak form of the above theorem:

\begin{theorem}\label{thm:weak-vanishing}
  In the settings of the above theorem, in part (2) assume instead
  that $H$ is nef and that there exists an ample divisor $M$ on $X$
  such that for every image $V$ of a stratum of $(Y,B)$ with $\dim
  V\ge c$, one has $V\cdot M^c\cdot H^{\dim V-c}>0$. 

  Then $R^p\pi _*(R^q f_* \OO _Y(L))=0$ for any $p>c$ and $q\geq 0$.
\end{theorem}
\begin{proof}
  Let $D_X$ be a general element in the linear
  system $|nM|$, with $nM$ very ample. Denote $D_Y:= f\inv D_X$
  and $L_{D_Y}= (L+D_Y)|_{D_Y}$. 
  Consider the short exact sequence
  \begin{displaymath}
    0 \to \cO_Y(L) \to \cO_Y(L+D_Y) \to \cO_{D_Y}(L_{D_Y}) \to 0.
  \end{displaymath}
  In this sequence the line bundle $L_{D_Y}$ is
  $\bQ$-linearly equivalent to $(K_Y + D_Y + B + f^*H)|_{D_Y}= K_{D_Y}
  + B\cap D_Y + f^*H|_{D_Y}$, and so is of the same nature as $L$ but for the
    smaller global embedded simple normal crossing pair $(D_Y,B\cap D_Y)$.

    The images of the strata of $(D_Y,B\cap D_Y)$ are strictly smaller
    than the images of the strata of $(Y,B)$. Part (1) of the above
    Theorem~\ref{t-tor} implies that in the
    long exact sequence the connecting homomorphisms 
    $R^q f_*\cO_{D_Y}(L_{D_Y}) \to R^{q+1} f_*\cO_{Y}(L)$ are
    zero. Indeed, the sections in the image are generically zero at
    the image of every stratum of $(Y,B)$. Thus, for every $q$ we have
    a short exact sequence
  \begin{displaymath}
    0 \to R^q f_*\cO_Y(L) \to R^qf_*\cO_Y(L+D_Y) 
    \to R^qf_*\cO_{D_Y}(L_{D_Y}) \to 0.
  \end{displaymath}
  Since $H+D_X$ is ample, for the middle term we have
  $R^p\pi_*(R^qf_*\cO_Y(L+D_Y))=0$ for $p>0$ by \eqref{t-tor}(2). 
  Thus, $R^p\pi_*(R^qf_*\cO_{D_Y}(L_{D_Y}))=0$ for $p>k$ implies 
  $R^p\pi_*(R^q f_*\cO_Y(L))=0$ for $p>k+1$. 

  Cutting $c$ times by general divisors in $|nM|$, we arrive at the
  situation of \eqref{t-tor}(2), which gives vanishing of $R^p\pi_*$ for $p>0$.
  For the original sheaves $R^q f_*\cO_Y(L)$, this gives vanishing of
  $R^p\pi_*$ for $p>c$. 
\end{proof}

 \begin{definition}\label{d-qlc} A {\bf qlc} variety is a variety $X$, a $\Q$-Cartier divisor $\omega$ on $X$ and a finite collection $\{ C \}$ of irreducible reduced subvarieties of $X$ such that there is a proper morphism $f:Y\to X$ from a global embedded simple normal crossings pair $(Y,B)$ 
 such that:
 \begin{enumerate}
 \item $f^*\omega \sim _\Q K_Y+B$ and $B$ is a boundary divisor,
 \item $\OO _X \cong f_* \OO _Y(\rup -(B_Y^{<1}) .)$,
 \item $\{ C \}$  is given by the images of the strata of $(Y,\rdown B.)$.
 \end{enumerate}
 The elements of $\{ C \}$ are the {\bf qlc centers} of the {\bf qlc pair} $[X,\omega ]$.
 \end{definition}
 \begin{proposition}\label{p-qlc1} If $(X,B)$ is a lc pair, then $X$ is a qlc pair with $\omega =K_X+B$ and $\{ C \}$ is given by $X$ and the non-klt centers of $(X,B)$.
 \end{proposition}
 \begin{proof}
 \cite[3.31]{Fujino09}.
 \end{proof}
 \begin{proposition}\label{p-qlc2} Let $[X,\omega ]$ be a qlc pair, $X'$
 be a union of qlc centers of $[X,\omega ]$ with the reduced scheme structure and $\mathcal I_{X'}\subset \mathcal O_X$ the corresponding ideal. Then
 \begin{enumerate}
 \item $[X', \omega '=\omega |_{X'}]$ is a qlc pair whose centers are the centers of $[X,\omega ]$ contained in $X'$.
 \item If $X$ is projective and $L$ is a Cartier divisor on $X$ such that $L-\omega $ is ample, then $$H^q(\OO _X(L))=H^q(\mathcal I _{X'}\otimes\OO _X(L))=0\qquad {\rm for}\ q>0.$$ In particular $H^q (\OO _{X'}(L))=0$ for $q>0$ as  $[X', \omega ']$ is a qlc pair and $L|_{X'}-\omega '$ is ample.
 \end{enumerate}
 \end{proposition}
 \begin{proof}
 \cite[4.4]{AmbroQLvars}, \cite[3.39]{Fujino09}.
 \end{proof}

\section{Non-rational centers}
\label{sec:centers}

\subsection{The ``Get rid of the $A$'' trick}
\label{sec:resolution-of-pairs}

Let $f:Y\to X$ be a log resolution of $(X,B)$. We write
$$K_Y+E-A+\Delta=f^*(K_X+B)$$ where
 $E\geq 0$ is reduced, $A\geq 0$ is integral and exceptional, and
 $\rdown \Delta .=0$. Here, 
$E$ has no common components with either $A$ or $\Delta$, but $A$ and
$\Delta$ may have common components.

\begin{theorem}\label{t-res} Let $(X,B)$ be a lc pair. Then there exist morphisms
$f':Y'\to X$ and $\nu :Y\to Y'$ such that:
\begin{enumerate}
\item $f=f'\circ \nu$ is a log resolution of $(X,B)$,
\item $Y'$ is normal, $E'=\nu _* E$, $\Delta '=\nu _* \Delta $, and
  $A'=\nu_* A=0$,
\item $(Y', E'+\Delta ')$ is dlt,
\item $\nu$ is an isomorphism at the generic point of each non-klt center of $(Y,E+\Delta )$ and in particular there is a bijection between the non-klt centers of $(Y,E+\Delta )$ and the non-klt centers of $(Y',E'+\Delta ')$.
\end{enumerate}\end{theorem}
\begin{proof}

\begin{claim}\label{c-1}
  For any resolution $Y_1\to X$, there exists a rational map $\alpha
  :Y_1\dasharrow Y'$ and a morphism $f':Y'\to X$ such that $\alpha _*
  A_1=0$, $K_{Y'}+\nu _* (E_1+\Delta_1 )$ is dlt and $K_{Y'}+\nu _*
  (E_1+\Delta_1)=(f')^*(K_X+B)$.
\end{claim}
\begin{proof} 
  By \cite{BCHM} we may run a $(K_{Y_1}+\Delta_1 )$-MMP over $X$ say
  $\beta :Y_1\dasharrow Y_2$. Let $f_2:Y_2\to X$ be the corresponding
  morphism, $E_2=\beta _* E_1$, $A_2=\beta _* A_1$ and $\Delta_2=\beta _*
  \Delta_1$.  Since $$E_2-A_2+(K_{Y_2}+ \Delta_2 ) =(f_2)^*(K_X+B),$$
  $K_{Y_2}+ \Delta_2$ is nef over $X$ and $(f_2) _*(E_2-A_2)\geq 0$, by the
  Negativity Lemma, $E_2 -A_2\geq 0$.  As $A_2$ and $E_2$ have no common
  components, $A_2=0$. Therefore $$K_{Y_2}+ E_2+\Delta_2
  =(f_2)^*(K_X+B),$$ $(Y_2,E_2+\Delta_2)$ is lc with the same non-klt
  places as $(X,B)$ and $Y_2$ is $\Q$-factorial.

  Let $\gamma :Y_3\to Y_2$ be a log resolution of
  $(Y_2,E_2+\Delta_2)$. We write
  $$K_{Y_3}+\Gamma =\gamma ^* (K_{Y_2}+E_2+\Delta _2)+F$$ 
  where $\Gamma$ and $F$ are effective with no common components, $\gamma
  _*(\Gamma )= E_2+\Delta_2$ and $\gamma_*F=0$.  Let $C$ be an effective
  $\gamma$-exceptional divisor such that $-C$ is ample over $Y_2$ and $||
  C ||\ll 1$. Let $H\sim _{\Q,Y_2}-C$ be a general ample $\Q$-divisor
  such that $K_{Y_3}+\Gamma +H$ is dlt.  We have that
$$\Gamma +H \sim _{\Q,Y_2}\Gamma -C \sim _{\Q,Y_2}\Xi$$ 
where $(Y_3,\Xi )$ is klt.  After running a $(K_{Y_3}+\Xi)$-MMP over
$Y_2$, we obtain $\mu:Y'\to Y_2$ such that $K_{Y'}+\Xi '$ is $\gamma$-nef, where $\Xi'$ denotes the strict transform of $\Xi$ on $Y'$ and similarly for other divisors.  Since
$$C'-F'+(K_{Y'}+\Xi ')\sim _{\Q , Y_2} 0$$ and $\mu _* (C'-F' )\geq 0$, then
by the Negativity Lemma, we have that $C'-F'\geq 0$. Since $|| C ||\ll
1$, this implies that $F'=0$ so that $K_{Y'}+\Gamma '=\mu ^*
(K_{Y_2}+E_2+\Delta_2)$.  Since a $(K_{Y_3}+\Xi)$-MMP over $Y_2$ is
automatically a $(K_{Y_3}+\Gamma +H)$-MMP over $Y_2$, it follows that
$K_{Y'}+\Gamma' +H'$ is dlt.  In particular $K_{Y'}+\Gamma'$ is dlt.
\end{proof}
\begin{claim}\label{c-2} There exists a resolution
  $\nu:Y\to Y'$ which is an isomorphism at the generic
  point of any non-klt center of $(Y,E+\Delta )$.
\end{claim}
\begin{proof} This follows from the characterization of dlt
  singularities given in  \cite{Szabo}.
\end{proof} 

Theorem~\ref{t-res} now follows. \end{proof}
\begin{remark} Note that variants of \eqref{c-1} appear in
  \cite[3.1]{KK10} and \cite[10.4]{Fuj10}, where they are referred to
  as an unpublished theorem of Hacon.
\end{remark}
\begin{corollary}\label{c-res1} 
  Let $Y$ be as in Theorem~\ref{t-res}. Then 
  one has $R^if_*\cO_Y=R^if_*\cO_Y(A)$ for all $i\geq 0$.  
\end{corollary}
\begin{proof} 
  We have  $\nu _* \OO
  _Y(A)=\OO _{Y'}$, and for $i>0$ one has $R^i\nu _* \OO _Y(A)=~0$ at the generic
  point of any non-klt center of $(Y',E'+\Delta ')$ and so by
  \eqref{t-tor}, $R^i\nu _* \OO _Y(A)=~0$. Therefore, 
$R^\bullet\nu_* \OO_Y(A) = \OO_{Y'}$. 

Since $(Y',B')$ is dlt, by
  \eqref{t-klt} $Y'$ has rational singularities, so 
$R^\bullet\nu_* \OO_Y = \OO_{Y'}$. 
It follows that $R^\bullet f_* \OO_Y(A) = R^\bullet f_* \OO_Y$ and so
$R^if_* \OO _Y=R^if_* \OO _Y(A)$ for all~$i$.
\end{proof}

\begin{proof}[Proof of Theorem~\ref{thm:centers}]
  Let $f\colon (Y, E+\Delta)\to X$ be a log resolution of $(X,B)$, as in
  Theorem~\ref{t-res}. Thus, $(Y,E+\Delta)$ is a normal crossing pair,
  and the $f$-images of the strata of $(Y,E+\Delta)$ are the non-klt
  centers of $(X,B)$.

  By \eqref{cor:assoc-primes}, the zero locus of any associated prime
  of $R^if_*\cO_Y(A)$ is a non-klt center. We are now done by the
  above Corollary~\ref{c-res1}.
\end{proof}

\begin{theorem}\label{thm:remove-A-always}
  For \emph{any} log resolution $f:Y\to X$ one has $R^\bullet
  f_*\cO_Y(A) = R^\bullet f_*\cO_Y$. 
\end{theorem}
\begin{proof}
  Let $f_1:Y_1\to X$, $f_2:Y_2\to X$ be two log resolutions. First, we
  are going to construct modifications $g_1:Y'_1\to Y_1$, $g_2:Y'_2\to
  Y_2$ by blowing up \emph{the strata of $E_1$, resp. $E_2$, only.} If
  $g_1$ is a sequence of such blowups then
  \begin{displaymath}
    R^\bullet (g_1)_* \cO_{Y'_1}(A'_1) = \cO_{Y_1}(A_1) 
    \quad \text{and} \quad
    R^\bullet (g_1)_* \cO_{Y'_1} = \cO_{Y_1}.
  \end{displaymath}
  The second equality follows because $Y_1$ is nonsingular. For the
  first equality, note that over the generic point of each stratum of
  $E_1$ one has $\cO_{Y'_1}(A'_1)=\cO_{Y'_1}$. Indeed, in this case
  $A'_1$ is the strict preimage of $A_1$, and $E_1\cup\Supp A_1$ is a
  normal crossing divisor.  Therefore, $R^i (g_1)_* \cO_{Y'_1}(A'_1) =
  0$ for $i>0$ at the generic point of each stratum of $E_1$. Hence,
  by \eqref{t-tor}, $R^i (g_1)_* \cO_{Y'_1}(A'_1) = 0$ for $i>0$.

  By making such sequences of blowups $g_1,g_2$, we can assume that
  $Y'_1, Y'_2$ have the same places of non-klt singularities of
  $(X,B)$ and that the birational map $\phi:Y'_1 \dasharrow Y'_2$ is an
  isomorphism at the generic point of each stratum of $E'_1$
  and $E'_2$. 

  Now by Hironaka there exists a sequence of blowups $h_1:\wY\to Y'_1$
  and a regular map $h_2:\wY\to Y'_2$ resolving the indeterminacies of 
  $\phi$. If the blowups $h_1$ are performed only inside the
  nonregular locus of $\phi$, as it can always be done, then $h_1,h_2$
  are isomorphisms at the generic point of each stratum of
  $\wE$. Applying \eqref{t-tor} again, we get 
  \begin{displaymath}
    R^\bullet (h_k)_* \cO_{\wY}(\wA) = \cO_{Y'_k}(A'_k) 
    \quad \text{and} \quad
    R^\bullet (h_k)_* \cO_{\wY} = \cO_{Y'_k} 
    \quad \text{for } k=1,2.
  \end{displaymath}
  Putting this together, we get 
  \begin{displaymath}
    R^\bullet (f_k\circ g_k\circ h_k)_* \cO_{\wY}(\wA) = 
    R^\bullet (f_k)_* \cO_{Y_k}(A_k)
    \quad \text{for} \quad k=1,2.
  \end{displaymath}
  Since $f_1\circ g_1\circ h_1 = f_2\circ g_2\circ h_2$, we get 
  \begin{displaymath}
    R^\bullet (f_1)_* \cO_{Y_1}(A_1) = R^\bullet (f_2)_* \cO_{Y_2}(A_2) 
    \quad \text{and} \quad
    R^\bullet (f_1)_* \cO_{Y_1}      = R^\bullet (f_2)_* \cO_{Y_2}.
  \end{displaymath}
  Now if we choose $Y_1$ to be as in Corollary~\ref{c-res1} then we get
  the same conclusion for any other resolution $Y_2$. 
\end{proof}

\subsection{A resolution separated into levels}
\label{sec:horizontal-resolution}

For any $l\geq 0$, let $E'_{\geq l}$ (resp. $E'_{= l}$ and $E'_{\leq
  l}$) be the sum of the components of $E'$ whose image via $f':Y'\to X$ has
dimension at least (resp. equal to and less or equal to) $l$. We use a
similar notation for $E$.

\begin{proposition}\label{c-res} In Theorem~\ref{t-res}
  we may assume that the dimension of the image via $f$ of any stratum
  of $E_{\geq l}$ is at least $l$.
\end{proposition}
\begin{proof} Let $\mu :\tilde Y \to Y$ be a log resolution of $(Y,E+\Delta)$ and of any non-klt center $V$ of $(X,B)$.
We then have that the dimension of the image via $\tilde f=f\circ \mu $ of any stratum of $\tilde E_{\geq l}$ is at least $l$. 
Notice in fact that if this is not the case, then there are divisors $F_1,\ldots,F_k$ given by components of $\tilde {E}_{\geq l}$ such that 
$W=F_1\cap \ldots \cap F_k$ is a non-klt center of $(\tilde{Y},\tilde{E}+\tilde{ \Delta})$ with $\dim \tilde{f}(W)<l$. But then as $\tilde{f}(W)$ is a non-klt center, $W$ is contained in a component of $\tilde {E}$ different from $F_1,\ldots,F_k$.
As $\tilde E$ has simple normal crossings, this is impossible.

We must now show that there are morphisms $\tilde \nu : \tilde Y \to \tilde Y '$ and $\eta : \tilde Y ' \to Y'$ such that $\nu \circ \mu= \eta\circ\tilde \nu$, $\tilde \nu _* (\tilde E +\tilde \Delta )=\tilde E '+\tilde \Delta' $ where $(\tilde Y ' ,\tilde E '+\tilde \Delta' )$ is dlt, $\tilde A'=0$ and 
that $\tilde \nu$ is an isomorphism at the generic point of any non-klt center of  $(\tilde Y  ,\tilde E +\tilde \Delta )$.

Let $U$ be an open subset of $Y'$ which is isomorphic to an open subset of $Y$ such that $(E'+\Delta ')|_U$
has simple normal crossings support and for any divisor $F$ exceptional over $Y'$ with center contained in $Z=Y'-U$, we have $a(F,Y',E'+\Delta ')>-1$.
We may assume that if $G\subset \tilde {Y}$ is a $\mu$-exceptional divisor such that $a (G,Y, E+\Delta )>-1$, then $\mu (G)\subset Z$.
For any $0<\epsilon \ll 1$, we write $\nu ^*(K_{Y'}+(1-\epsilon)(E'+\Delta ')) +F=K_Y+\Gamma$
where $F$ and $\Gamma$ are effective with no common components and $\nu _* (\Gamma )=(1-\epsilon)(E'+\Delta ')$ and we let $K_{\tilde Y}+\tilde \Gamma =\mu ^*\nu ^*((K_{Y'}+(1-\epsilon)(E'+\Delta ')) +F)+\tilde F '$.
We now run a $(K_{\tilde Y}+\tilde \Gamma )$-MMP over $Y'$ to obtain $\eta : \tilde Y ' \to Y'$.

Since over $U$, $F$ and $\tilde F'$ are zero, we have that 
$\tilde \nu$ is an isomorphism over $U$.
There is a neighborhood of $Z$ over which $(Y',E'+\Delta ')$ is klt. It follows that $(\tilde Y',\tilde E '+\tilde \Delta ')$ is dlt.
After blowing up centers over $Z$, we may assume that $\tilde \nu $ is a morphism.
\end{proof}

\section{The sheaves $R^if_*\cO_Y$ and the proof of the main theorem}

\subsection{Leray spectral sequence}
\label{sec:leray}

\begin{lemma}\label{lem:leray}
  Let $X$ be a projective normal variety of dimension $\ge n$ and $f:Y\to X$ a resolution of singularities.  Assume
  that the sheaves $R^if_*\cO_Y$, $i>0$, are $C_{n-1-i}$.  Then $X$
  is $C_{n}$.
\end{lemma}
\begin{proof}
  Let $L$ be an ample sheaf on $X$.
  Consider the Leray spectral sequence
  \begin{displaymath}
    E_2^{p,q} = H^p(R^qf_* \cO_Y(-sL)) \Rightarrow
    H^{p+q}(\cO_Y(-sf^*L)), \quad 
    \text{for some }s\gg0. 
  \end{displaymath}  
  Since $f^*L$ is big and nef, the limit
  $E_{\infty}^k = H^k(\cO_Y(-sf^*L))$ is zero for $k<n\le\dim X$ by
  the Kawamata-Viehweg vanishing theorem.

  By the assumption, we have $E_2^{p,q}=0$ for $p+q\le n-2$ and $q>0$.
  Inspecting this spectral sequence we easily conclude that
  $E_2^{p,0}=E_{\infty}^{p,0}$ for $p\le n-1$.  On the other hand, we
  have $E_{\infty}^{p,0}\subset E_{\infty}^{p}$ and the latter is zero
  for $p\le n-1$. So $H^p(\cO_X(-sL))=0$ for $p\le n-1$ and $\cO_X$ is
  $C_n$ by~\eqref{lem:Hartshorne}.
\end{proof}

\begin{remark}
  For the $C_3$ case, \cite[3.1]{AlexeevLimits} gives a necessary and
  sufficient condition: $X$ is $C_3$ $\Leftrightarrow$ $R^1f_*\cO_Y$ is $C_1$.
\end{remark}

To prove Theorem~\ref{thm:main}, it is now sufficient to
prove that $H^p(X, R^qf_*\cO_Y(-sL))=0$ for $q>0$, $p+q\le d$
and
$s\gg0$. The rest of this section will be devoted to establishing this
fact.

\subsection{Vanishing theorems for  unions of centers}
\label{sec:vanishing-for-centers}

Let $(X,B)$ be a lc pair and let $f':Y'\to X$ and $\nu :Y\to Y'$ be as
in Section~\ref{sec:horizontal-resolution}. In particular, we may
assume that $Y$ and $Y'$ satisfy \eqref{t-res} and \eqref{c-res}.
\begin{lemma}\label{l-0} Let $Z$ and $W$ be unions of non-klt centers of $(Y,E+\Delta )$ and let 
$Z'=\nu (Z)$, $W'=\nu (W)$.  
Then $$R^\bullet \nu _*\OO _{Z}(-W)=R^\bullet \nu _*\OO
_{Z}(A-W)=\OO_{Z'}(-W'),$$
where $\OO _{Z'}(-W')$ is the ideal sheaf of $W'\cap Z'$ in $Z'$.
\end{lemma}
\begin{proof} If $Z$ is irreducible, then $(Z,(K_Y+E+\Delta)|_Z)$ is a global embedded simple normal crossings pair.
For $j>0$, $R^j \nu _*\OO _{Z}(A)=0$ at the generic point of any non-klt
center of $(Y',E'+\Delta' )$ (cf. \eqref{t-res}, \eqref{c-2}) 
and hence $R^j \nu _*\OO _{Z}(A)=0$ by \eqref{t-tor}. Since $[Z',(K_{Y'}+E'+\Delta')|_{Z'}]$ is a qlc variety cf. \eqref{p-qlc2},
we also have $\nu _*\OO _{Z}(A)=\OO _{Z'}$ cf. (2) of \eqref{d-qlc}.

We proceed by induction on $d$ the maximum dimension of a component of $Z$,
the case $d=0$ being obvious. 
For a fixed $d$, we 
proceed by induction on the number of components of $Z$.
If $Z$ is irreducible, then as $(Y',E'+\Delta ')$ is dlt, $Z'$ is normal with rational singularities and
$R^\bullet \nu _* \OO _Z\cong \OO _{Z'}\cong R^\bullet \nu _* \OO _Z(A)$. By induction on $d$, $R^\bullet \nu _*\OO _{Z\cap W}\cong R^\bullet \nu _*\OO _{Z\cap W}(A)\cong \OO _{Z'\cap W'}$.
Pushing forward the short exact sequence $$0\to \OO _Z(A-W)\to \OO _Z(A)\to \OO _{Z\cap W}(A)\to 0$$ and noticing that $$\nu _* \OO _Z=\nu _* \OO _Z(A)\cong \OO _{Z'}\to \nu _* \OO _{Z\cap W}=\nu _* \OO _{Z\cap W}(A)\cong \OO _{Z'\cap W'}$$ is surjective,
it follows that $R^\bullet \nu _* \OO _Z(-W)\cong R^\bullet \nu _* \OO _Z(A-W)\cong \OO _{Z'}(-W')$.

If $Z$ is not irreducible, then let $Z_0$ be an irreducible component of $Z$ and write $Z=Z_0+Z_1$ where $Z_1$
is the union of the components of $Z$ distinct from $Z_0$.
Consider the short exact sequence
$$0\to \OO _{Z_0}(A-W-Z_1)\to \OO _{Z}(A-W )\to \OO _{Z_1 }(A-W )\to 0.$$
By induction on $d$ the number of components of $Z$, we have $R^\bullet \nu _* \OO _{Z_1 }(-W )\cong R^\bullet \nu _* \OO _{Z_1 }(A-W )\cong\OO _{Z'_1 }(-W' )$. By what we have shown above, $R^\bullet \nu _*\OO _{Z_0}(-W-Z_1)
\cong R^\bullet \nu _*\OO _{Z_0}(A-W-Z_1)
\cong \OO _{Z'_0}(-W'-Z'_1)$. It follows that $R^\bullet \nu _* \OO
_Z(-W)\cong R^\bullet \nu _* \OO _Z(A-W)\cong \OO _{Z'}(-W')$. The
assertion is proved.
\end{proof}
Let $L$ be an ample line bundle on $X$.
\begin{lemma}\label{l-1} If $X$ is projective, then $H^j(\OO _{Y'}(-s(f')^*L))=0$ for all $j<n$ and $s>0$.
\end{lemma}
\begin{proof} Since $(Y', \nu _*(E+\Delta)=E'+\Delta ')$ is dlt, it has rational singularities and so $R^\bullet \nu _* \OO _Y\cong \OO _{Y'}$. Therefore, by Serre duality, $$H^j(\OO _{Y'}(-s(f')^*L))\cong H^j(\OO _{Y}(-sf^*L)) \cong H^{n-j}(\OO _{Y}(K_Y+sf^*L))^\vee$$ and the lemma follows from Kawamata-Viehweg vanishing cf. \cite[4.3.7]{Lazarsfeld}.
\end{proof}

\begin{lemma}\label{l-2} Assume that $X$ is projective.
  Let $0\leq F\leq E$ be a reduced divisor such that the minimum of
  the dimension of the images of any stratum of $F$ in $X$ is $\geq
  k$. Let $Z$ and $W$ be unions of non-klt centers of $(Y,F)$.  Then
$$H^l(\OO _{Z}(-W-sf^*L))\cong H^l(\OO _{Z}(A-W-sf^*L))\cong H^l(\OO _{Z'}(-W'-s(f')^*L))=0$$ for all $l\leq k-1$ and any $s>0$. 
\end{lemma}
\begin{proof} Since $R^\bullet \nu _*\OO _{Z}(-W)\cong R^\bullet \nu _*\OO _{Z}(A-W)\cong \OO _{Z'}(-W')$
  cf. \eqref{l-0}, it suffices to prove that $H^l(\OO
  _{Z}(-W-sf^*L))=0$.  If $Z$ is irreducible, then it is a smooth
  variety.  We have $(f^*L|_Z)^k\neq 0$ and $f^*L|_Z$
  is nef so that by Kawamata-Viehweg vanishing (cf.
  \cite[4.3.7]{Lazarsfeld}), we have $$h^l(\OO _{Z}(-sf^*L))=0\qquad
  \forall l\ \leq k-1 .$$ In general, the proof is by induction on the
  maximal dimension of a component of $Z$ and on the number of
  components of $Z$.  When $\dim Z=0$, there is nothing to prove.
  If $Z$
  is irreducible, then the statement follows from the short exact
  sequence
  $$ 0\to \OO _Z(-W)\to \OO _Z \to \OO _{Z\cap W}\to 0 $$
  and induction on $\dim Z\cap W$.

If $Z$ is not irreducible, then we let $Z_0$ be an irreducible component of $Z=Z_0\cup Z_1$ and we consider the short exact sequence
$$0\to \OO _{Z_0}(-W-Z_1)\to \OO _Z (-W)\to \OO _{Z_1}(-W)\to 0.$$
The statement now follows by induction on the number of components and what we have shown above.
\end{proof}
\begin{lemma}\label{l-3} If $X$ is projective and $K_X+B$ is Cartier and every non-klt center of $(X,B)$ has dimension $\geq d$,
then
  $H^i(R^jf_*\mathcal O _{E_{=k}}(A-E+E_{\geq k}-sf^*L))=0$ for $i+j\leq
  d-1$   and $s>0$ (resp. for $i+j=d$, 
  $j>0$ and $s>0$).
\end{lemma}

The proof given below is fairly technical. We use the prepared resolution of
Section~\ref{sec:horizontal-resolution}, previously established
vanishing results, and a multilevel induction. At the heart of the argument,
however, is the method of Koll\'ar introduced in
\cite{KollarDirectImages2}, which gives vanishing and torsion free
theorems for the sheaves $R^j f_* \omega_Y$ by using variation of pure
Hodge structures, and the extension of this method to variation of
mixed Hodge structures in \cite{KawamataFiberSpaces, KawamataVMHS,
  Kaw11}. 

\begin{proof} We will beging by proving the required vanishing for
  $i+j\leq d-1$ and $s>0$.  Let $V_1,\ldots , V_\tau$ be the
  irreducible non-klt centers of $(X,B)$ of dimension $k$ and let
  $Z_t$ be the union of the components of $E_{=k}$ that dominate
  $V_t$. Note that $E-E_{\geq k}=E_{\leq k-1}$. Let $Z_{\geq
    t}=Z_t+\ldots +Z_\tau$.  We have short exact sequences 

  \begin{multline}\tag{$\star$}0\to \OO
  _{Z_t}(A-E_{\leq k-1}-Z_{\geq t+1})\to \OO _{Z_{\geq t}}(A-E_{\leq
    k-1})\\ \to \OO _{Z_{\geq t+1}}(A-E_{\leq k-1})\to 0.\end{multline} 
  Note that by
  \eqref{c-res}, $Z_t\cap Z_{\geq t+1}=\emptyset $, so that the sequences  $(\star )$ are split and we have the equalities $(A-E_{\leq
    k-1}-Z_{\geq t+1})|_{Z_t}\sim _{X}K_{Z_t}+E_{\geq k+1}|_{Z_t}$.
Since the
  above short exact sequence $(\star )$
is split, it remains exact (and split) after applying $R^jf_*$ and
  twisting by $-sL$.  Thus, it suffices to show that $H^i(R^jf_*\mathcal O
  _{Z_t}(A-E_{\leq k-1}-Z_{\geq t+1}-sf^*L))=0$ for $i+j\leq d-1$ and
  $s>0$.

  Let $Z=Z_t$ and $V=V_t$.  We may assume that there are resolutions
  $g:\tilde Z \to Z$ and $h:\tilde V\to V$ and there is a snc divisor
  $\Xi$ on $\tilde V$ such that $\tilde f:\tilde Z\to \tilde V$ is a
  morphism which is smooth over $\tilde V-\Xi$ and the same holds for
  any stratum of $ g^{-1}_* ( E_{\geq k+1}|_{Z})$ (cf. see for
  example the proof of \cite[2]{Kaw11} for a similar statement). We
  may also assume that $g$ is an isomorphism at the generic point of
  the strata of $(Z, E_{\geq k+1}|_{Z})$ so that $F:=K_{\tilde Z}+
  g^{-1}_* (E_{\geq k+1}|_{Z})- g^*(K_Z+E_{\geq k+1}|_{Z})$ is
  effective and $g$-exceptional.  Note also that if $W$ is any stratum
  of $(Z, E_{\geq k+1}|_{Z})$, then by the same argument, $F|_{\tilde
    W}$ is $g|_{\tilde W}$-exceptional.  We may further assume that
  $g$ and $h$ are given by sequences of blow ups along smooth centers
  and thus that they are induced by morphisms which (by abuse of
  notation) we also denote by $h:\tilde X\to X$ and $g:\tilde Y\to Y$.
  We write $K_{\tilde Y}+\tilde E-\tilde A=(f\circ g)^*(K_X+B)$. Note
  that by construction $\tilde Y\to \tilde X$ is an isomorphism over
  the complement of a proper closed subset of $V$ and so it is easy to
  see that we have $\tilde E _{\geq k+1}|_{\tilde Z}=g^{-1}_* (E_{\geq
    k+1}|_{Z})$.

  Let $M:=(A-E_{\leq k-1})|_{Z}$ and $\tilde M=g^*M+F\sim _X K_{\tilde
    Z}+\tilde E_{\geq k+1}|_{\tilde Z}$.  Since $F$ is $g$
  exceptional, $g_*\tilde M =M$.  
  By \eqref{t-tor}, $R^\bullet g_*\OO_{\tilde Z}(\tilde M ) \cong g_*\OO_{\tilde Z}(\tilde M)=\OO _Z(M)$ 
  and so it suffices to show that $H^i(R^j(f\circ g)_*\mathcal O
  _{\tilde Z}(\tilde M-s(f\circ g)^*L))=0$ for $i+j\leq d-1$ and
  $s>0$.

By \eqref{t-tor} $R^ih_* R^j\tilde f_* \mathcal O
  _{\tilde Z}(\tilde M )=0$ for $i>0$, and thus it suffices to show
  that $H^i(R^j\tilde f_* \mathcal O _{\tilde Z}(\tilde M-s(f\circ g)^*L))=0$ for
  $i+j\leq d-1$ and $s>0$.

To this end, let $ \tilde Z=\tilde  Z^1+ \ldots  +\tilde  Z^p$  and for any
multi-index $I=\{ i_1,\ldots , i_l \}\subset \{ 1,\ldots , p \}$ let
$\tilde Z^I=\tilde Z^{i_1}\cap \ldots \cap \tilde Z^{i_l}$ and $\tilde Z[I]=\tilde Z^{i_1}+\ldots
+\tilde Z^{i_l}$. 
Similarly let $\tilde E_{\geq k+1}=\tilde G_1+\ldots +\tilde G_\rho$ and for any
multi-index $I=\{ i_1,\ldots , i_l \}\subset \{ 1,\ldots , \rho \}$
let $\tilde G^I=\tilde G_{i_1}\cap \ldots \cap \tilde G_{i_l}$ and $\tilde G[I]=\tilde G_{i_1}+
\ldots + \tilde G_{i_l}$.  We use a similar notation for $Z$ and $ E_{\geq k+1}$.
It suffices to show that:
\begin{enumerate}
\item The sheaves $R^j\tilde f_* \mathcal O _{\tilde Z
  }(\tilde M )$ admit filtrations whose quotients are direct sums of
  summands of sheaves of the form $R^j\tilde f_* \mathcal O _{\tilde
    Z^L\cap \tilde G^N}(\tilde M-\tilde G[\bar N]-\tilde Z[\bar L])$
  where $N\subset \{ 1,\ldots ,\rho\}$, $L\subset \{ 1,\ldots ,p\}$,
  $\bar N=\{ 1,\ldots , \rho \} \setminus N$ and $\bar L =\{1,\ldots
  , p\}\setminus L$.
\item $H^i(R^j\tilde f_* \mathcal O _{\tilde Z^L\cap \tilde G^N}(\tilde  M-\tilde G[\bar
  N]-\tilde Z[\bar L]-s(f\circ g)^*L))=0$ for $i+j\leq d-1$ and $s>0$. \end{enumerate}

To see the first statement, notice that $(\tilde M-\tilde G[\bar
N]-\tilde Z[\bar L])|_{\tilde Z^L\cap \tilde G^N}\sim _{X} K_{\tilde Z^L\cap \tilde  G^N}$ and
hence the sheaves $R^j\tilde f_* \mathcal O _{\tilde Z^L\cap \tilde G^N}(\tilde M-\tilde G[\bar N]-\tilde Z[\bar L])$ are obtained as upper
canonical extensions of the bottom pieces in the Hodge filtration of a
variation of {\it pure} Hodge structures
cf. \cite{KollarDirectImages2}. Since pure Hodge structures are a semisimple category, these sheaves split as a direct sum of simple summands.

Consider now two disjoint sets $I,J\subset\{1,\dotsc,\rho \}$, 
$I\cap J=\emptyset$, and an index $\alpha \in \overline { I \cup
  J}$ in the complement of $I\sqcup J$. 
Let $I'=I\cup \alpha$ and $J'=J\cup \alpha$.

We have short exact sequences
\begin{multline*}\tag{$\sharp$}0\to \OO_{\tilde Z\cap \tilde G^{I}}(\tilde M-\tilde G[J'])\to \OO _{\tilde Z\cap \tilde G^I }(\tilde M-\tilde G[J])
\to \OO _{\tilde Z\cap \tilde G^{I'} }(\tilde M-\tilde G[J])\to 0.\end{multline*}
Proceeding by ascending induction on $|I|+|J|$ we may assume that the
claim (1) holds for $R^j\tilde f_*$ of the right and left hand sides of the above
short exact sequence.  Note that $(\tilde M-\tilde G[J])|_{\tilde Z\cap \tilde G^I
}\sim _{X }K_{\tilde Z\cap \tilde G^I }+\tilde G[\overline {J+I}]|_{\tilde Z\cap \tilde G^I }$, and so
by \cite[5.1]{KawamataVMHS} (see also \cite[15]{Kaw11}),
 we have that the sheaves $R^j\tilde f_*\OO
_{\tilde Z\cap \tilde G^I }(\tilde M-\tilde G[J])$ are obtained as upper
canonical extensions of the bottom pieces in the Hodge filtration of a
variation of {\it mixed} Hodge structures. Since pure Hodge structures
are a semisimple category,  a morphism of mixed Hodge structures to a simple Hodge structure is either surjective or trivial.
Pushing forward $(\sharp )$ and using the filtration on $R^j\tilde  f_*\OO
_{\tilde Z\cap \tilde G^{I'} }(\tilde M -\tilde G[J])$ first and then the filtration on $R^j\tilde f_*\OO
_{\tilde Z\cap \tilde G^I }(\tilde M-\tilde G[J'])$ (both filtrations exist by our inductive hypothesis), we obtain the required filtration on $R^j\tilde f_*\OO
_{\tilde Z\cap \tilde G^I }(\tilde M-\tilde G[J])$.
The claim now follows (once the base of the
induction has been verified).

We must now verify the base of the induction, i.e. that the claim holds for 
sheaves of the form $R^j\tilde f_*\OO _{\tilde Z\cap \tilde G^N }(\tilde M-\tilde G[\bar N])$ where $N\subset\{1,\dotsc,\rho \}$.
Note that $\tilde Z=\tilde Z[1,2,\ldots ,p]$.
Recall that $(\tilde M-\tilde G[\bar N])|_{\tilde Z\cap \tilde G^N }\sim _{X }K_{\tilde Z\cap \tilde G^N }$  
and so $(\tilde M-\tilde Z [\overline I]-\tilde G[\bar N])|_{\tilde Z[I]\cap \tilde G^N }\sim _{X }K_{\tilde Z[I]\cap \tilde G^N }$.   Consider the short exact sequences
$$0\to \mathcal O _{V}(-V\cap W)\to \mathcal O_{V+W}\to \mathcal O_{W}\to 0,$$
$$0\to \mathcal O _{W}\to \mathcal O_{W}(V)\to \mathcal O_{W\cap V}\to 0,$$
where $V=\tilde Z ^\alpha$, $W=\tilde Z[I]\cap \tilde Z^J\cap \tilde
G^N$, and index sets $I,J,\alpha,I',J'\subset\{1,\dotsc,p \}$ defined as above.

Tensoring the above sequences by the
line bundles $\mathcal O _{\tilde Z}(\tilde M-\tilde G [\overline
N]-\tilde Z [\overline{I'+J}])$ and $\mathcal O _{\tilde Z}(\tilde
M-\tilde G [\overline N]-\tilde Z [\overline{I+J}])$, we obtain (up to
linear equivalence over $X$) the following short exact sequences
$$0\to \omega_{\tilde Z[I]\cap\tilde Z^{J'}\cap \tilde G^N}\to \omega_{\tilde Z[I']\cap \tilde Z^J\cap \tilde G^N}\to \omega_{\tilde Z[I]\cap \tilde Z^J\cap \tilde G^N}(\tilde Z^\alpha)\to 0,$$
$$0\to \omega _{\tilde Z[I]\cap \tilde Z^J\cap \tilde G^N}\to \omega _{\tilde Z[I]\cap \tilde Z^J\cap \tilde G^N}(\tilde Z^\alpha)\to \omega _{\tilde Z[I]\cap \tilde Z^{J'}\cap \tilde G^N}\to 0.$$
By the above arguments, using the higher direct images of these exact sequences, and proceeding by descending induction on $|I|$, we reduce to the case that $|I|=1$ i.e. to the case $K_{\tilde Z^J\cap \tilde G^N}\sim _{X}(\tilde M-\tilde Z[\bar J]-\tilde G[\bar N])|_{\tilde Z^J\cap \tilde G^N}$ and the first statement follows.

We now prove the second statement. By \eqref{l-2} we have $H^i(\mathcal O _{Z^L\cap G^N}(M-G[\bar N]-Z[\bar L]-sf^*L))=0$ for $i\leq d-1$ and $s>0$. Notice that $g$ is an isomorphism on an open subset of ${\tilde Z^L\cap \tilde G^N}$, thus, by \eqref{t-tor}, $$R^\bullet (g|_{\tilde Z^L\cap \tilde G^N})_*\OO_{\tilde Z^L\cap \tilde G^N}(\tilde M-\tilde G[\bar N]-\tilde Z[\bar L])\cong (g|_{\tilde Z^L\cap \tilde G^N})_*\OO_{\tilde Z^L\cap \tilde G^N}(\tilde M-\tilde G[\bar N]-\tilde Z[\bar L]).$$
Moreover, since $g|_{\tilde Z^L\cap \tilde G^N}$ is birational and $F|_{\tilde Z^L\cap \tilde G^N}$ is $g|_{\tilde Z^L\cap \tilde G^N}$-exceptional, we have that  $$(g|_{\tilde Z^L\cap \tilde G^N})_*\mathcal O _{\tilde Z^L\cap \tilde G^N}(\tilde M-\tilde G[\bar N]-\tilde Z[\bar L])=\mathcal O _{Z^L\cap G^N}(M-G[\bar N]-Z[\bar L]).$$ 
In particular $H^i(\mathcal O _{\tilde Z^L\cap \tilde G^N}(\tilde M-\tilde G[\bar N]-\tilde Z[\bar L]-s(f\circ g)^*L))=0$ for $i\leq d-1$ and $s>0$.

By \cite{KollarDirectImages2}, it follows that
\begin{multline*}R^\bullet \tilde f_*\mathcal O_ {\tilde Z^L\cap \tilde G^N}(\tilde M-\tilde G[\bar N]-\tilde Z[\bar L])=
\sum R^i\tilde f_*\mathcal O_ {\tilde Z^L\cap \tilde G^N}(\tilde M-\tilde G[\bar N]-\tilde Z[\bar L])[-i]\end{multline*} and so 
$H^i(R^j\tilde f_*\mathcal O_ {\tilde Z^L\cap \tilde G^N}(\tilde M-\tilde G[\bar N]-\tilde Z[\bar L])-s(f\circ g)^*L)=0$ for $i+j\leq d-1$.
 
We will now prove the required vanishing for $i+j= d$, $k>d$ 
  and $s>0$.
Note that in this case we have a short exact sequence
$$0\to \mathcal O _{Z^L\cap G^N}(M-G[\bar N]-Z[\bar L])\to \mathcal O _{Z^L\cap G^N}(M+E_{=d}-G[\bar N]-Z[\bar L])$$ $$\to \mathcal O _{E_{=d}\cap Z^L\cap G^N}(M+E_{=d}-G[\bar N]-Z[\bar L])  \to 0.$$
By \eqref{l-2} we have $H^i(\mathcal O _{Z^L\cap G^N}(M+E_{=d}-G[\bar N]-Z[\bar L])-sf^*L))=0$ for $i\leq d$ and $s>0$ and $H^i(\mathcal O _{E_{=d}\cap Z^L\cap G^N}(M+E_{=d}-G[\bar N]-Z[\bar L])-sf^*L))=0$ for $i\leq d-1$ and $s>0$. 
It follows that $H^i(\mathcal O _{Z^L\cap G^N}(M-G[\bar N]-Z[\bar L])-sf^*L))=0$ for $i\leq d$ and $s>0$.
The required vanishing now follows from the proof of the previous case $i+j\leq d-1$.

Finally, we will now prove the required vanishing for $i+j= d$, $k=d$ 
  and $s>0$.
It suffices to show that $H^{d-j}(R^jf_*\OO _{Z}(A-sf^*L))=0$ for $j>0$ and $s>0$.
Following the above arguments, it suffices to check that if $(\tilde G[\bar N]+\tilde Z[\bar L])|_{\tilde Z^L\cap \tilde G^N}\ne 0$, then $$H^{d}(\mathcal O _{\tilde Z^L\cap \tilde G^N}(\tilde M-\tilde G[\bar N]-\tilde Z[\bar L]-s(f\circ g)^*L))=0\qquad {\rm for} \ s>0$$ and if $(\tilde G[\bar N]+Z\tilde [\bar L])|_{\tilde Z^L\cap \tilde G^N}=\emptyset $, then  
$$H^{d}(\mathcal O _{\tilde Z^L\cap \tilde G^N}(\tilde M-s(f\circ g)^*L))\cong H^d(\tilde f_* \mathcal O _{\tilde Z^L\cap \tilde G^N}(\tilde M-s(f\circ g)^*L))\qquad {\rm for} \ s>0.$$ Note that as $Z^L\cap  G^N\cap( G[\bar N]+ Z[\bar L])$ is a union of non-klt centers, it is seminormal cf. \cite[9.1]{Fuj10}.
Similarly to what we have seen above, we have $$R^\bullet g_* \mathcal O _{\tilde Z^L\cap \tilde G^N\cap(\tilde G[\bar N]+\tilde Z[\bar L])}(\tilde M)=\mathcal O _{Z^L\cap  G^N\cap( G[\bar N]+ Z[\bar L])}( M)$$ and so
$$R^\bullet g_* \mathcal O _{\tilde Z^L\cap \tilde G^N}(\tilde M-\tilde G[\bar N]-\tilde Z[\bar L])=\mathcal O _{Z^L\cap  G^N}(M- G[\bar N]- Z[\bar L]).$$
We also have $R^\bullet g_* \mathcal O _{\tilde Z^L\cap \tilde G^N}(-\tilde G[\bar N]-\tilde Z[\bar L])=\mathcal O _{Z^L\cap  G^N}(- G[\bar N]- Z[\bar L]).$
Since $M=A$, by \eqref{l-0} we have 
\begin{multline*}H^{d}(\mathcal O _{\tilde Z^L\cap \tilde G^N}(\tilde M-\tilde G[\bar N]-\tilde Z[\bar L]-s(f\circ g)^*L))=H^{d}(\mathcal O _{Z^L\cap G^N}(M-G[\bar N]-Z[\bar L]-sf^*L))\\
^{\eqref{l-0}}\cong H^{d}(\mathcal O _{Z^L\cap G^N}(-G[\bar N]-Z[\bar L]-sf^*L))\cong H^{d}(\mathcal O _{\tilde Z^L\cap \tilde G^N}(-\tilde G[\bar N]-\tilde Z[\bar L]-s(f\circ g)^*L)).\end{multline*}

Let $f=\dim (\tilde Z^L\cap \tilde G^N)-\dim \tilde V$.  We
have \begin{multline*}H^{d}(\mathcal O _{\tilde Z^L\cap \tilde
    G^N}(-\tilde G[\bar N]-\tilde Z[\bar L]-s(f\circ
  g)^*L))\\^{\rm (Serre\ duality)}\cong H^f(\omega _{\tilde Z^L\cap \tilde G^N}(\tilde G[\bar N]+\tilde Z[\bar L]+s(f\circ g)^*L))^\vee\\
  ^{\eqref{t-tor}}\cong H^0(R^f\tilde f_* \omega _{\tilde Z^L\cap
    \tilde G^N}(\tilde G[\bar N]+\tilde Z[\bar L])\otimes \OO _{\tilde
    V}(sh^*L))^\vee\end{multline*} If $(\tilde G[\bar N]+\tilde Z[\bar
L])|_{\tilde Z^L\cap \tilde G^N}\ne 0$, then $R^f\tilde f_* \omega
_{\tilde Z^L\cap \tilde G^N}(\tilde G[\bar N]+\tilde Z[\bar L])$ is
torsion free and generically $0$ and hence vanishes
(cf. \eqref{t-tor}).  If, on the other hand, $(\tilde G[\bar N]+\tilde
Z[\bar L])|_{\tilde Z^L\cap \tilde G^N}=\emptyset $, then $R^f\tilde
f_* \omega _{\tilde Z^L\cap \tilde G^N}\cong \omega _{\tilde V}$
(cf. \cite{KollarDirectImages}).  By Serre duality and \eqref{t-tor}, we
have
\begin{multline*} H^0( \omega _{\tilde V}(sh^*L))^\vee 
\cong ^{\rm S.D.}H^d(\OO _{\tilde V}(-sh^*L))\cong ^{\eqref{t-tor}} H^d(\tilde f_* \mathcal O _{\tilde Z^L\cap \tilde G^N}(-s(f\circ g)^*L)).\end{multline*}\end{proof}

\subsection{The structure of the sheaves $R^if_*\cO_Y$}
\label{sec:structure}

By the Kawamata-Viehweg vanishing theorem, we have that
$R^if_*\cO_Y(A-E)=0$ for $i>0$ since $A-E = K_Y +\Delta - f^*(K_X+B)$.
We will now build up the sheaves $R^if_*\cO_Y\cong R^if_*\cO_Y(A)$, going from $R^if_*\cO_Y(A-E)$ to
$R^if_*\cO_Y(A)$ by adding the parts $E_{=l}$ defined in
Section~\ref{sec:horizontal-resolution} one by one.

\begin{proof}[Proof of \eqref{thm:main}]
By \cite{HX11}, we may assume that $(X,B)$ is projective. Adding a sufficiently ample divisor to $B$ we may assume that $K_X+B$ is ample, and so we may assume that $m(K_X+B)\sim H$ is a general very ample divisor (for some integer $m>0$).
Let $\nu :X'\to X$ be the corresponding normal cyclic cover (cf. \cite[5.20]{KollarMoriBook}) and $K_{X'}+B '=\nu ^*(K_X+B)$. 
Then $(X',B')$ is log canonical, $K_{X'}+B'$ is Cartier and the non-klt centers are given by the inverse images of the non-klt centers of $(X,B)$ cf. \cite[5.20]{KollarMoriBook}.
Note that if $Y'=Y\times _X X'$, then $f':Y'\to X'$ is a resolution and $\mu :Y'\to Y$ is a finite map so that 
$R^i\mu _* \OO _{Y'}=0$ for $i>0$ and $\OO _Y$ is a direct summand of $\mu _*\OO _{Y'}$. 
Thus it is easy to see that if $X'$ is $C_{d+2}$ then so is $X$ (similarly if $R^if'_*\OO _{Y'}$ is $C_{d+1-i}$ for $i>0$, then so is $R^if_*\OO _Y$).
Thus, replacing $(X,B)$ by $(X',B')$, we may assume that $K_X+B$ is Cartier. 

 Recall that by \eqref{thm:remove-A-always}, we have $R^jf_* \mathcal O _Y\cong R^jf_* \mathcal O _Y(A)$. For $\dim X -1\geq k \geq d$, consider the short exact sequences
$$0\to R^jf_*\mathcal O_Y(A-E+E_{\geq k+1})\to R^jf_*\mathcal O_Y(A-E+E_{\geq k})\to R^jf_*\mathcal O_{E=k}(A-E+E_{\geq k})\to 0.$$ (Note that these sequences are exact by \eqref{t-tor}; moreover $E_{\geq \dim X-1}=0$.)
By \eqref{l-3} we have $H^i(R^jf_*\mathcal O_{E=k}(A-E+E_{\geq k}- sf^*L))=0$ for $i+j\leq d$, $j>0$ and $s>0$.
Since $R^jf_*\mathcal O_Y(A-E)=0$ for $j>0$, it follows that 
$0=H^i( R^jf_*\mathcal O_Y(A-E- sf^*L)))\to H^i(R^jf_*\mathcal O_Y(A- sf^*L)))$ is surjective for $i+j\leq d$ and $j>0$ and so $R^jf_*\mathcal O_Y\cong R^jf_*\mathcal O_Y(A)$ is $C_{d-j+1}$ for $j>0$. By \eqref{lem:leray}, $X$ is $C_{d+2}$.
\end{proof}

\bibliographystyle{amsalpha}

\def\cprime{$'$} \def\cprime{$'$}
\providecommand{\bysame}{\leavevmode\hbox to3em{\hrulefill}\thinspace}
\providecommand{\MR}{\relax\ifhmode\unskip\space\fi MR }
\providecommand{\MRhref}[2]{%
  \href{http://www.ams.org/mathscinet-getitem?mr=#1}{#2}
}
\providecommand{\href}[2]{#2}

\end{document}